\documentclass[11pt]{amsart}

\usepackage{amsmath}
\usepackage{amssymb}
\usepackage{amsfonts}

\newtheorem{theorem}{Theorem}
\theoremstyle{plain}

\newtheorem{claim}{Claim}

\newtheorem{corollary}{Corollary}

\newtheorem{lemma}{Lemma}

\newtheorem{problem}{Problem}

\numberwithin{equation}{section}

\def\be{\begin{equation} }
\def\ee{\end{equation} }
\def\<{\langle}
\def\>{\rangle}

\begin{document}
\title[Isometric deformations of minimal surfaces in $S^{4}$]
{Isometric deformations of minimal \\surfaces in $S^{4}$}
\author{Theodoros Vlachos}
\address{Department of Mathematics, University of Ioannina, 45110 Ioannina,
Greece}
\email{tvlachos@uoi.gr}

\begin{abstract}
We consider the isometric deformation problem for oriented non simply connected immersed minimal surfaces  $f:M \to S^{4}$.
We prove that the space of all isometric minimal  immersions of  $M$   into $S^{4}$ 
with the same normal curvature function is, within congruences, either finite or a circle.
 Furthermore, we show that 
for any compact immersed minimal surface  in $S^{4}$ with nontrivial normal bundle there are at most finitely many 
noncongruent immersed minimal surfaces in $S^{4}$ isometric to it with the same normal curvature function.
\end{abstract}

\maketitle

\section{Introduction}
A classical question about isometric  immersions is to decide if given an isometric immersion $f:M \to N$, this is, up to ambient isometries,
 the unique way of immerse  isometrically the Riemannian manifold $M$ into the Riemannian manifold $N$. When $f$ is a minimal immersion, one can ask 
if this is the unique isometric minimal immersion  of $M$ into $N$, up to ambient isometries. If this the case, $f$ is called \textit{minimally rigid}. 
The rigidity aspects of minimal hypersurfaces in a Euclidean space or in a sphere have drawn several author's attention. 
A conclusive result due to Dajczer and Gromoll \cite{DG} states that a complete minimal hypersurfaces in
 $\mathbb{R}^{n+1}$ or in $S^{n+1}, n \geq 4,$ is minimally rigid (in the Euclidean case the assumption that the hypersurface 
  does not  contain $\mathbb{R}^{n-3}$ as a factor is needed).

This result fails to hold in general for surfaces.  
However, Choi, Meeks and White \cite{CMW} proved that a properly embedded minimal surface in 
$\mathbb{R}^{3}$ with more than one end is minimally rigid. The case where the Euclidean space is replaced by a sphere is more difficult.
 A result due to Barbosa \cite{B} says that a minimally immersed
 2-sphere in a sphere is minimally rigid, while Ramanathan \cite{R} has shown that for each compact surface minimally immersed in $S^3$, 
there are only finitely many noncongruent immersed  minimal surfaces isometric to it.

It is interesting in minimal surface theory to determine whether a given minimal surface can be deformed in a nontrivial way.
The oldest known example is the deformation of the catenoid into the helicoid.  We are interested in isometric deformations of immersed minimal surfaces $f:M \to S^4$ which preserve the normal curvature function. 
If $M$ is simply connected, then there exists a smooth $2\pi$-periodic parameter deformation
 $f_{\theta}$ of $f$,  the \textit{associated family},
through  isometric  minimal immersions with the same normal curvature function. The procedure is to rotate the second fundamental form of $f$ by $\theta$
  and then integrate the system of Gauss, Codazzi and Ricci equations (cf. \cite{TG}). The members of this family are
 noncongruent, unless $f$ is superminimal. We recall that superminimal surfaces are the minimal surfaces whose Hopf differential vanishes  identically, 
or equivalently, 
the curvature ellipse is a circle at each point. 

Thus the rigidity for simply connected minimal surfaces fails in a natural way, and consequently the rigidity problem for minimal surfaces
has a global nature. The above  procedure cannot be carried out in general
 in the presence of nontrivial fundamental group $\pi_1(M)$. In this case, the obstruction is described as a homeomorphism from $\pi_1(M)$
 to the isometry group 
$\text{Isom}(S^4)$  of  $S^4$.

Our aim is to study the space of isometric minimal immersions of $M$ into $S^4$ with  the same normal curvature function. 
For any connected minimal surface, not necessarily compact, we have the following result that was  inspired by a recent paper due to 
 Smyth and Tinaglia \cite{ST}.

\begin{theorem}\label{Th1}
 Let $f:M \to S^4$ be an isometric minimal immersion of an oriented connected 2-dimensional Riemannian manifold
 $M$ into $S^4$ with normal curvature function $K_{N}$. Then, within congruences, the space of all isometric minimal immersions of $M$ into $S^4$
 with the same normal curvature function $K_{N}$
 is either finite or a circle.
\end{theorem}

The main result deals with the number of isometric minimal immersions of compact surfaces with the same normal curvature function. 
We prove that, within this kind of deformations, 
 compact  immersed minimal surfaces in $S^4$ with nontrivial normal bundle are minimally rigid up to finiteness.

\begin{theorem} \label{Th2}
 Let $f:M \to S^4$ be an isometric minimal  immersion of a compact oriented 2-dimensional Riemannian manifold
 $M$ into $S^4$ with nontrivial normal bundle and normal curvature function $K_{N}$.
 Then there exist at most finitely many 
noncongruent minimal immersions of $M$ into $S^4$ with the same normal curvature function $K_{N}$.
\end{theorem}

As an application, we show that  if a compact minimal surface allows a one parameter group of
 intrinsic isometries that preserve the normal curvature function,
 then these isometries extend to extrinsic ones.
\begin{corollary}
 Let $f:M \to S^4$ be an isometric minimal immersion of a compact oriented 2-dimensional Riemannian manifold
 $M$ into $S^4$ with nontrivial normal bundle. Suppose that $M$ 
admits a one parameter group of isometries $\varphi_t:M \to M, t \in \mathbb{R},$  that preserve the orientation and the normal curvature function. 
Then there exists a one parameter group of isometries $\tau_t:S^4 \to S^4$ such that $f\circ\varphi_t=\tau_t \circ f$ for all $t \in \mathbb{R}.$
\end{corollary}

The paper is organized as follows: In section 2, we fix the notation and
give the local theory of minimal surfaces in $S^4$. In section 3, we discuss the associated family of simply connected minimal surfaces.
 In section 4,  we deal with the deformability of not necessarily simply connected minimal surfaces and prove Theorem 1. 
Finally, in the last section we give the proof of the main result, namely Theorem 2.

\section{Local theory of minimal surfaces in $S^4$}
Let $f:M\rightarrow S^{4}$ be an immersed minimal surface, i.e., an isometric minimal immersion
of a connected oriented $2$-dimensional Riemannian manifold $M$,
 with  normal bundle $Nf$ and second fundamental form $B^f=B$.

Let $\{e_{j}\}$ be a local orthonormal frame field on $S^{4}$, which agrees with the orientation of $TS^{4}$, and let $\{\omega
_{j}\}$ be the coframe dual to $\{e_{j}\}$. The structure equations of $S^{4}$ are 
\begin{eqnarray}\label{structure}
d\omega _{j} &=&\sum_{k}\omega _{jk}\wedge \omega _{k}, \\
d\omega _{jk} &=&\sum_{l}\omega _{jl}\wedge \omega _{lk}-\omega _{j}\wedge
\omega _{k},
\end{eqnarray}%
where the connection form $\omega _{jk}$ is given by $\omega
_{jk}(X)=\langle \tilde{\nabla }_{X}e_{j},e_{k}\rangle $, $\tilde{\nabla}$ is the Levi-Civit\'{a} connection with
respect to  the standard metric $\langle
.,.\rangle $ of $S^{4}$. We choose the frame such
that, restricted to $M$, $e_{1}$ and $e_{2}$ are tangent and agree with the given orientation  of $M$ and consequently $e_{3 }$ and $e_{4}$
are normal to the surface. Then we have $\omega _{\alpha }=0$. By (2.1) and
Cartan's Lemma, we get 
\begin{equation*}
\omega _{j\alpha }=\sum_{k}h_{jk}^{\alpha }\omega _{k},{\ }h_{jk}^{\alpha
}=h_{kj}^{\alpha }, {\ }{\ } j,k \in \{1,2\}, {\ }{\ } \alpha \in \{3,4\}.
\end{equation*}%
The assumption that $f$ is minimal is equivalent to $h_{11}^{\alpha
}+h_{22}^{\alpha }=0$. Restricting equations (2.1) and (2.2) to $M,$ we get
the Cartan structure equations of $f.$

We may view $M$ as a Riemann surface 
with complex structure determined as usual by the metric and the orientation. 
The complexified tangent bundle $TM\otimes \mathbb{C}$ is decomposed into
the eigenspaces of the complex structure $J$, called $T^{\prime }M$ and $%
T^{\prime \prime }M$, corresponding to the eigenvalues $i$ and $-i$. 
The second fundamental form $B$, which takes values in $Nf$, can
be complex linearly extended to $TM\otimes \mathbb{C}$ with values in the
complexified vector bundle $Nf\otimes \mathbb{C}$ and then decomposed
into its $(p,q)$-components, $p+q=2,$ which are tensor products of $p$ many 
1-forms vanishing on $T^{\prime \prime }M$ and $q$ many 1-forms vanishing on $%
T^{\prime }M$. The  minimality of $f$
implies that  the $(1,1)$-component of $B$
vanishes. Consequently, for a local complex
coordinate $z=x+iy$, we have the following decomposition 
\begin{equation*}
B=B^{(2,0)}+B^{(0,2)},
\end{equation*}
where 
\begin{equation*}
B^{(2,0)}=B\big(\dfrac{\partial}{\partial z},\dfrac{\partial}{\partial z} \big)dz^2,{\ }{\ } B^{(0,2)}=\overline{B^{(2,0)}} {\ }{\ }{\hbox {and}}{\ }{\ }
\dfrac{\partial}{\partial z} ={\frac{1}{2}}\big({\frac{\partial }{\partial x}}-i{\frac{\partial 
}{\partial y}}\big ).
\end{equation*}

The Hopf differential is  by definition the differential form of type $(4,0)$ 
\begin{equation*}
\Phi:=\langle B^{(2,0)},B^{(2,0)}\rangle.
\end{equation*}

The components of the second fundamental form are given by 
\begin{equation*}
h_{1}^{\alpha }:=h_{11}^{\alpha}=\langle
B(e_{1},e_{1}),e_{\alpha }\rangle, {\ } h_{2}^{\alpha }:=h_{12}^{\alpha}=\langle
B(e_{1},e_2),e_{\alpha }\rangle,
\end{equation*}%
where $\alpha =3$ or $4$. We use complex
vectors, and we put 
\begin{equation*}
H_{\alpha }=h_{1}^{\alpha }+ih_{2}^{\alpha },\text{ }E=e_{1}-ie_{2}{\ \ }
\text{and }\varphi =\omega _{1}+i\omega _{2}.
\end{equation*}

If we choose a local complex coordinate $z=x+iy$ such that $\varphi =\rho dz,$ for some smooth real function $\rho$,
then  we easily get 
\begin{equation*}
\Phi ={\frac{1}{4}}\big( {\overline{H}_{3}^{2}}+{\overline{H}%
_{4}^{2}} \big) \varphi ^{4}.
\end{equation*}

The \textit{curvature ellipse},  for each point $%
p$ in $M$, is the subset of the normal space $N_{p}f$ given by 
\begin{equation*}
\mathcal{E} (p)=\big\{B(X,X):X\in T_{p}M,|X|=1%
\big\}.
\end{equation*}%
It is known that $\mathcal{E}(p)$ is indeed an
ellipse (possibly degenerated). The zeros of $\Phi $ are precisely the points where 
the curvature ellipse is a circle. The minimal surface is called \textit{superminimal} if $\Phi$
 is identically zero, or equivalently, if the curvature ellipse is a circle at each point.

It is a consequence of the structure equations that the Hopf differential is holomorphic (cf. \cite{Ch,ChW}). 
Hence either a minimal surface is superminimal, or the points where
the curvature ellipse is a circle are isolated.

The \textit{normal curvature} function $K_{N}$ (cf. \cite{A}), which depends on the orientation of the normal bundle $Nf$,
is defined by the equation
\begin{equation*}
 d\omega_{34}=-K_{N}\omega_1 \wedge \omega_2,
\end{equation*}
or equivalently,
\begin{equation}\label{normalcur}
K_{N}=i\left( H_{3}{\overline{H}_{4}}-{\overline{H}_{3}}%
H_{4}\right).
\end{equation}
It is not hard to verify that
\begin{equation*}
|K_{N}|={\frac{2}{\pi }}{\hbox {Area}}(\mathcal{E}).
\end{equation*}%
 Let $\kappa \geq \mu
\geq 0$ be the length of the semi-axes of the curvature ellipse. Then%
\begin{equation*}
|K_{N}|=2\kappa \mu.
\end{equation*}%
The length of the second fundamental form is given by%
\begin{equation}\label{lengthB}
\left\Vert B\right\Vert ^{2}=2\big(|{H_{3}}|^{2}+|{H_{4}}|^{2}%
\big).
\end{equation}%

\smallskip

\begin{lemma}\label{local theory}
  Assume that  $f:M\rightarrow S^{4}$ is not superminimal and let $M_1$ be the set of isolated points where the curvature ellipse is a circle.
 Around each  point in $M\smallsetminus M_1$, there exist a local complex coordinate $(U,z)$ with
 $U  \subset     M\smallsetminus M_1$ and 
orthonormal frames  $\{e_1, e_2\}$ in $TM|U$,
 $\{e_3, e_4\}$ in $Nf|U$ which agree with the given orientations such that 

(i) the Riemmannian metric of $M$ is given by 
$$ds^2=\dfrac{|dz|^2}{(\kappa_1^2 - \mu_1^2)^{1/2}} {\ }{\ }{\ }  \text{and}  {\ }{\ }{\ }
\frac{\partial}{\partial z} = { E \over 2(\kappa_1^2 - \mu_1^2)^{1/4}},$$

(ii) 
  $e_3$ \text{and} $e_4$ give respectively the directions of the major 
and the minor axes of the curvature ellipse, and

(iii)
 $H_3  ={\kappa_1}, H_4= i {\mu_1 }$, where $\kappa_1$ and $\mu_1$ are smooth real valued functions with 
$$\kappa=|{\kappa_1}|, {\ }  \mu=|{\mu_1 }|.$$ 

Moreover, the connection and the normal connection forms,  with respect to this frame, are given by
\begin{equation}\label{connection}
 \omega_{12}= -\frac{1}{4}*d\log\ ({\kappa}_1^2 -{\mu }_1^2 ),
 {\ }\omega_{34} = *\frac{{\kappa_1} d{\mu_1 }-{\mu_1 } d{\kappa_1}}{{\kappa}_1^2 -{\mu }_1^2},
\end{equation}
where $*$ stands for the Hodge operator.
\end{lemma}
\begin{proof}
 Assume that  the curvature ellipse  is not a
circle at a point $x \in M$. By continuity the same property holds on an 
 open set $U_1$ around $x$. Let $\{e_{3}, e_{4}\}$
be an arbitrary  oriented   orthonormal frame in $N f|U_1$.
 We  introduce Cartesian coordinates $(x_1,x_2)$ in
 each fiber of $Nf|U_1$
adapted to this frame. In view of 
$$B(e_1,e_1)= h_1^{3}e_{3}+h_1^{4}e_{4}   
{\ }{\ } \text{and} {\ }{\ } B(e_1, e_2)= h_2^{3}e_{3}+h_2^{4}e_{4},$$
 we  deduce that the quadratic equation of  the curvature ellipse is given by
\begin{eqnarray*}
\lefteqn{ |H_{4}|^2  x_1^2 -
 2{\hbox {Re}} (H_{3} \overline{H}_{4}   )x_1x_2
 + |H_{3}|^2   x_2^2 } \\
& & = |H_{3}|^2 |H_{4}|^2 - \big({\hbox {Re}}
 (H_{3} \overline{H}_{4}  ) \big)^2.
\end{eqnarray*}
  On $U_1$
we now choose the frame so that $e_{3}$ and $ e_{4}$ give the directions of the major and minor
 axes respectively (cf. \cite{AFR}). This means that the coefficient of $x_1x_2$ above
must vanish, i.e.,
 $H_{3} \overline{H}_{4} $ is imaginary. 
It is clear that the length of the semi-axes of  the curvature ellipse are
given by ${\kappa}  =  |H_{3} |{\ }
{\hbox {and}} {\ }  {\mu}= |H_{4} |$. Since ${\kappa}  =  |H_{3} |>0$ on $U_1$, we may choose $e_1$ and $ e_2$ so that $h_2^{3}=0$.
Then the fact that $H_{3} \overline{H}_{4} $ is imaginary implies that $h_1^{4}=0$. We set $\kappa_1:=h_1^3$ and $\mu_1:=h_2^4$. 

We consider local coordinates $(u,v)$ on a neighborhood $U\subset U_1$ of $x$ such that
\begin{equation*}
e_1 =
\frac{1}{r_1}\frac{\partial}{\partial u} {\ }{\ } \text{and} {\ }{\ } e_2 =
\frac{1}{r_2}\frac{\partial}{\partial v}.
\end{equation*}
 Then $\omega _1 = r_1 du$ and
 $\omega _2 = r_2 dv$. From (2.1) we find that
 $$\omega _{12} = -{(r_1) _v \over  r_2}du -{(r_2) _u \over r_1 }dv.$$
 Taking the exterior derivative of 
$$\omega _{13} = \kappa_1 r_1 du, {\ }
  \omega _{23} =- \kappa_1 r_2 dv, {\ } \omega _{14} = \mu_1 r_2 dv, 
{\ } \omega _{24} = \mu_1 r_1 du$$
 and using the structure equations (2.2), we obtain
$$2\kappa_1 (r_1) _v +(\kappa_1)_v r_1 + \mu_1 r_1 r_2 \omega _{34}(e_1) =0,$$
$$2\kappa_1 (r_2) _u +(\kappa_1)_u r_2 - \mu_1 r_1 r_2 \omega _{34}(e_2) =0,$$
$$2\mu_1 (r_2) _u +(\mu_1)_u r_2 - \kappa_1 r_1 r_2 \omega _{34}(e_2) =0,$$
$$2\mu_1 (r_1) _v +(\mu_1)_v r_1 + \kappa_1 r_1 r_2 \omega _{34}(e_1) =0.$$
Eliminating $\omega _{34}(e_1)$ and $\omega _{34}(e_2)$, we get
$$ 2(\kappa_1^2 - \mu_1^2)(r_2)_u +\big(\kappa_1 (\kappa_1)_u- \mu _1 (\mu _1)_u \big)r_2 = 0,$$
$$ 2(\kappa_1^2 - \mu_1^2)(r_1)_v +\big(\kappa_1 (\kappa_1)_v-
\mu _1 (\mu _1)_v \big)r_1 = 0,$$
and so
  $r_1 ^2 (\kappa_1^2 - \mu_1^2)^{1/2}$ depends only on  $u$ and
   $r_2 ^2 (\kappa_1^2 - \mu_1^2)^{1/2}$   depends only on  $v$. Then we
   introduce the complex coordinate $z=x+iy$ given by
   $$x = \int r_1 (\kappa_1^2 - \mu_1^2)^{1/4}  du {\ }{\ } {\hbox {and}} {\ }{\ }
  y = \int r_2  (\kappa_1^2 - \mu_1^2)^{1/4}  dv.$$
  Now it is easy to verify that
  $$ds^2 = { |dz|^2\over  (\kappa_1^2 - \mu_1^2)^{1/4}}   {\ }{\ }
 {\hbox {and}}
 {\ }{\ }  E = 2(\kappa_1^2 - \mu_1^2)^{1/4}\dfrac{\partial}{\partial z}.$$
This completes the proof of parts (i)-(iii).

Taking the exterior derivatives of 
$$\omega_{13}=\kappa_1\omega_1, {\ }\omega_{23}=-\kappa_1\omega_1, {\ }\omega_{14}=\mu_1\omega_2,{\ } \omega_{24}=\mu_1\omega_1$$ 
and using the structure equations, we obtain (\ref{connection}).
\end{proof}

\subsection{The splitting of the Hopf differential}
Using the null frame field 
$$\eta=e_{3}+ie_{4}, {\ } \bar  \eta =e_{3}-ie_{4}$$
 of the complexified normal bundle
$Nf\otimes \mathbb{C}$, we have 
\begin{equation*}
 \langle B^{(2,0)},B^{(2,0)}\rangle =
\langle B^{(2,0)},\eta\rangle \langle B^{(2,0)}, \bar{\eta} \rangle.
\end{equation*}
Therefore, from the definition of the Hopf differential, we get 
\begin{equation*}
\Phi =\dfrac{1}{4}\big( {\overline{H}_{3}^{2}}+{\overline{H}%
_{4}^{2}}\big) \varphi ^{4}=\dfrac{1}{4}k_{+}k_{-}\varphi^{4},
\end{equation*}
where 
\begin{equation*}
k_{\pm}:= {\overline{H}_{3}} \pm i{\overline{H}%
_{4}}.
\end{equation*}

The functions 
$$a_{\pm}:= |k_{\pm}|$$
are globally well-defined. Their geometric meaning is that they both determine the geometry of the  curvature ellipse. 
Indeed, since the Gaussian curvature $K$ of $M$ is given by 
$$K=1-\dfrac{1}{2}\left\Vert B\right\Vert ^{2},$$
 it follows from (\ref{normalcur}) and (\ref{lengthB}) that 
\begin{equation*}
a_{\pm}=(1-K \pm K_N)^{1/2}= \kappa {\pm}  \varepsilon\mu ,
\end{equation*}
where $\varepsilon= \pm 1$, according to  weather $K_N \geq 0$ or  $K_N \leq 0$.

We use the above mentioned notation throughout the paper.

\section{Associated family of minimal surfaces in $S^4$}

Let  $f:M \to S^4$ be an isometric minimal immersion of a simply connected oriented 2-dimensional Riemannian manifold $M$ 
with second fundamental form $B$ and normal bundle $Nf$. 
For each $\theta \in S^1=\mathbb{R}/2\pi \mathbb{Z}$, we consider the orthogonal and parallel tensor field 
$$J_{\theta}=\cos\theta I + \sin \theta J, $$
where  $I$ is the identity map and $J$ is the complex structure determined by the orientation and the metric of $M$. 
It is easy to see that the symmetric section $\varGamma_{\theta}$ of the homomorphism bundle $\text{Hom}(TM\times TM,Nf)$ given by 
$$\varGamma_{\theta}(X,Y):=B(J_{\theta}X,Y),$$
where $X$ and $Y$ are tangent to $M$, satisfies the Gauss, Codazzi and Ricci equations with respect to the normal connection $\nabla^{\perp}$ 
of $Nf$ (cf. \cite{TG}). Hence 
there exist 
an isometric immersion $f_{\theta}:M \to S^4$ and a  vector bundle isomorphism 
$$T_{\theta}:Nf \to Nf_{\theta},$$ 
which is parallel and orthogonal,
 such that 
$$B^{f_{\theta}}(X,Y)=T_{\theta}(B(J_{\theta}X,Y))$$ 
for all $X$ and $Y,$ where $B^{f_{\theta}}$ is the second fundamental form of $f_{\theta}$. Obviously, $f_{\theta}$ is also minimal.
The $2\pi$-periodic family $f_{\theta}$ is the \textit{associated family} of $f$. The members of the associated family are
  noncongruent, unless $f$ is superminimal and so each $f_{\theta}$ is congruent to $f$ (cf. \cite{ET}).

The normal curvature function of $f_{\theta}$
 coincides with the normal curvature function of $f$, where the orientation of $Nf_{\theta}$  is naturally induced from the
 orientation on $Nf$ and the bundle isomorphism $T_{\theta}$. Indeed, for  a local orthonormal frame $\{e_3, e_4\}$ of $Nf$,
 we consider the  frame of $Nf_{\theta}$ given by
\begin{equation*}
 e^{\theta}_{3}:=T_{\theta}(e_3), {\ }{\ } e^{\theta}_{4}:=T_{\theta}(e_4).
\end{equation*}
 Then it is easy to see that $H_3 , H_4$ and the corresponding functions $H^{\theta}_{3}, H^{\theta}_{4}$ for $f_{\theta}$
are related by
\begin{equation*}
 H^{\theta}_{3}=\exp (-2i\theta) H_3 {\ }{\ } \text{and} {\ }{\ } H^{\theta}_{4}=\exp(-2i\theta) H_4.
\end{equation*}
From these, by virtue of (\ref{normalcur}), it follows that $f$ and $f_{\theta}$ have the same normal curvature function.

Actually the associated family is the only way to isometrically deform any simply connected immersed minimal
surface in $S^4$ preserving the normal curvature function. This has already been proved by Eschenburg and Tribuzy in \cite{ET0}.

In order to state their result, we recall the notion of absolute
value type functions introduced in \cite{EGT,ET0}. A smooth complex valued function $u$ defined on a connected 
oriented surface $M$  is 
called of \textit{holomorphic type}  if locally $u=u_{0}u_1,$ where $u_{0}$ is holomorphic and $u_{1}$ is  smooth without zeros.
  A function $a:M\rightarrow
\lbrack 0,+\infty )$  on  $M$ is called of \textit{%
absolute value type}  if there is a function $u$  of holomorphic type on $M$  such that $a=|u|$. The zero set 
of such a function is either isolated
 or the whole of $M$, and outside its zeros the function is smooth.

\begin{theorem}\label{ET0} \cite{ET0}
Let $f:M\rightarrow S^{4}$ be an immersed minimal surface which is not nonsuperminimal with Gaussian curvature $K$ and normal curvature function $K_{N}$.
 Then  the functions
 $a_{\pm}= (1-K \pm K_N)^{1/2}$ are of absolute value type and satisfy 
\begin{equation}\label{Laplace}
\Delta \log a_{\pm} = 2K \mp K_N, 
\end{equation}%
where $\varDelta$ stands for the Laplace operator of $M$. Furthermore, if $M$ is simply connected, then
any other minimal immersion of $M$ into $S^4$ having the same normal curvature function $K_N$ is congruent to some $f_{\theta}$.
\end{theorem}

\section{Isometric deformations of minimal surfaces in $S^4$   preserving the normal curvature function}

We consider the following 
\begin{problem}
Given an immersed minimal surface $f:M\rightarrow S^{4}$  with normal curvature function $K^f_N=K_N$, 
 describe the space of all isometric minimal immersions of $M$ into $S^4$ with  the  same normal curvature function $K_N$. 
\end{problem}

From the holomorphicity of the Hopf differential we know that either  $f$ is superminimal or the curvature ellipse is a circle at isolated points only.
The answer to the above problem is already known in the case where $f$ is superminimal, 
since superminimal surfaces are rigid among superminimal surfaces (cf. \cite{J,V}).

Hereafter we assume that $f$ is not superminimal.  Let $g:M\rightarrow S^{4}$ be another immersed
 minimal surface with  the same normal curvature $K_N$. 
If $M$ is not simply connected, then we consider the Riemannian covering map  $p: \tilde{M} \to M$, $\tilde{M}$ being the universal cover of $M$ equipped 
with the metric and the orientation that makes $p$ an orientation preserving local isometry. Then the  immersed minimal surfaces
 $\tilde{f}:=f\circ p$ 
and $\tilde{g}:=g\circ p$ have the same normal curvature $\tilde K_N=K_N\circ p$. 
The orientation on each bundle $N\tilde{f}$ and $N\tilde{g}$ is naturally
induced from that of $Nf$ and $Ng$, respectively. According to Theorem \ref{ET0}, $\tilde{g}$ is congruent to some $\tilde{f}_{\theta}$
in the associated family of $\tilde{f}$. Now the question is whether $\tilde{f}_{\theta}$ projects to an isometric minimal immersion
 $f_{\theta}:M\rightarrow S^{4}$. 

Hence the study of the space of all isometric minimal immersions of $M$ into $S^4$ with the same normal curvature $K_N$ 
is reduced to the study of the set
\begin{equation*}
\mathcal{S}(f):=\left\{\theta \in [0,2\pi]:{\ } \text{there exists}{\ }  f_{\theta}:M\rightarrow S^{4} {\ } \text{so that} {\ }
 \tilde{f}_{\theta}= f_{\theta}\circ p \right\}. 
\end{equation*}
Obviously $0 \in \mathcal{S}(f)$. Moreover, for each $\theta \in \mathcal{S}(f)$, $f$ and ${f}_{\theta}$ have the same normal curvature,
where  the orientation of the normal bundle of $f_{\theta}$ 
is inherited in a natural way from the orientation of $N\tilde{f}_{\theta}$. Indeed, for any $x\in M,$ the normal curvature 
$K_N^{{f}_{\theta}}$ of ${f}_{\theta}$ is given by 
\begin{eqnarray*}
 K_N^{{f}_{\theta}}(x)&=& K_N^{{f}_{\theta}}\circ p(\tilde x)=K_N^{{f}_{\theta}\circ p}(\tilde x)
= K_{N}^{{\tilde f}_{\theta}}(\tilde x)\\&=&K_N^{\tilde f}(\tilde x)=K_N^{f}\circ p(\tilde x)=K_N(x),
\end{eqnarray*}
where $\tilde x \in p^{-1}(x).$

\begin{lemma}\label{deck}
 For any $\sigma$ in the group  $\mathcal{D}$ of deck transformations of the covering map $p: \tilde{M} \to M$, 
 the minimal immersions $\tilde{f}_{\theta}$ and
  $\tilde{f}_{\theta}\circ \sigma$ are congruent.
\end{lemma}
\begin{proof}
It is enough to prove the existence of an orthogonal and parallel isomorphism between  the normal bundles of $\tilde{f}_{\theta}$ and
  $\tilde{f}_{\theta}\circ \sigma$ that preserves the second fundamental forms. If $T_{\theta}$ is the isomorphism 
between  the normal bundles of $\tilde{f}$ and $\tilde{f}_{\theta}$, then we define the bundle isomorphism  
$$\mathit{\Sigma}_{\theta}:N\tilde{f}_{\theta} \to N(\tilde{f}_{\theta}\circ \sigma)$$
so that at any point $\tilde{x} \in \tilde{M}$  
$$\mathit{\Sigma}_{\theta}|_{\tilde{x}}:N_{\tilde{x}}\tilde{f}_{\theta} \to N_{\tilde{x}}(\tilde{f}_{\theta}\circ \sigma)$$
is given by 
$$\mathit{\Sigma}_{\theta}|_{\tilde{x}}(\xi):=
T_{\theta}|_{\sigma(\tilde{x})}\big(T^{-1}_{\theta}|_{\tilde{x}}(\xi)\big), {\ }{\ } \xi \in N_{\tilde{x}}\tilde{f}_{\theta}.$$
For any $\tilde{v}, \tilde{w} \in T_{\tilde{x}}\tilde{M}$ the second fundamental form of $\tilde{f}_{\theta}\circ \sigma$ is given by
\begin{eqnarray*}{}
B^{\tilde{f}_{\theta}\circ \sigma}|_{\tilde{x}}(\tilde{v}, \tilde{w})&=& 
B^{\tilde{f}_{\theta}}|_{\sigma(\tilde{x})}\big(d\sigma_{\tilde{x}}(\tilde{v}),d\sigma_{\tilde{x}} (\tilde{w})\big)\\
&=&T_{\theta}|_{\sigma(\tilde{x})}\big(  B^{\tilde{f}}|_{\sigma(\tilde{x})}\big(\tilde{J}_{\theta}\circ d\sigma_{\tilde{x}}(\tilde{v}),d\sigma_{\tilde{x}} (\tilde{w})\big) \big),
\end{eqnarray*}
where $B^{\tilde{f}_{\theta}}$ is the second fundamental form of $\tilde{f}_{\theta}$, 
$$\tilde{J}_{\theta}=\cos \theta \tilde{I}+\sin \theta \tilde{J}$$
 and
$\tilde{J}$ is the complex structure of $\tilde{M}.$ Since $\sigma$ is a deck transformation, we have
\begin{equation*}
 \tilde{f} \circ \sigma =\tilde{f} {\ }{\ } \text{and} {\ }{\ } \tilde{J}_{\theta}\circ d\sigma= d\sigma \circ \tilde{J}_{\theta}.
\end{equation*}
 Thus  it follows that
\begin{eqnarray*}{}
B^{\tilde{f}_{\theta}\circ \sigma}|_{\tilde{x}}(\tilde{v}, \tilde{w})
&=&T_{\theta}|_{\sigma(\tilde{x})}\big(  B^{\tilde{f}}|_{\sigma(\tilde{x})}\big( d\sigma_{\tilde{x}}\circ \tilde{J}_{\theta}(\tilde{v}),d\sigma_{\tilde{x}} (\tilde{w}) \big) \big)
\\
&=&T_{\theta}|_{\sigma(\tilde{x})}\big(  B^{\tilde{f}\circ \sigma}|_{\tilde{x}}\big(  \tilde{J}_{\theta}(\tilde{v}),\tilde{w} \big) \big),
\end{eqnarray*}
 or equivalently,
\begin{equation*}
B^{\tilde{f}_{\theta}\circ \sigma}|_{\tilde{x}}(\tilde{v}, \tilde{w}) =
T_{\theta}|_{\sigma(\tilde{x})}\big(  B^{\tilde{f}}|_{\tilde{x}}\big(  \tilde{J}_{\theta}(\tilde{v}),\tilde{w} \big) \big).
\end{equation*}
Then bearing in mind  the definition of $\mathit{\Sigma}_{\theta}$ and the relation 
\begin{equation*}
 B^{\tilde{f}_{\theta}}|_{\tilde{x}}(\tilde{v}, \tilde{w}) 
=T_{\theta}|_{\tilde{x}}\big(  B^{\tilde{f}}|_{\tilde{x}}\big(  \tilde{J}_{\theta}(\tilde{v}),\tilde{w} \big) \big),
\end{equation*}
we find
\begin{equation*}
 B^{\tilde{f}_{\theta}\circ \sigma}|_{\tilde{x}}(\tilde{v}, \tilde{w}) =
\mathit{\Sigma}_{\theta}\big( B^{\tilde{f}_{\theta}}|_{\tilde{x}}(\tilde{v}, \tilde{w})  \big),
\end{equation*}
which shows that $\mathit{\Sigma}_{\theta}$ preserves the second fundamental forms of $\tilde{f}_{\theta}$ and
  $\tilde{f}_{\theta}\circ \sigma$.

Now let $\xi=T_{\theta}(\eta)$ be an arbitrary section of $N\tilde{f}_{\theta}$, where $\eta$ is a section of $N\tilde{f}.$ Then 
\begin{equation*}
\mathit{\Sigma}_{\theta}(\xi)= T_{\theta}(\eta \circ \sigma^{-1})\circ \sigma
\end{equation*}
and consequently for any $\tilde{X}$ tangent to $\tilde{M}$ we have
\begin{eqnarray*}
(\nabla^{\perp}_{\tilde{X}}\mathit{\Sigma}_{\theta})\xi&=&
\nabla^{\perp}_{\tilde{X}}\mathit{\Sigma}_{\theta}(\xi)- \mathit{\Sigma}_{\theta}(\nabla^{\perp}_{\tilde{X}}\xi)\\
&=&\nabla^{\perp}_{\tilde{X}}\big(  T_{\theta}(\eta \circ \sigma^{-1})\circ \sigma  
 \big)-T_{\theta}\big(\nabla^{\perp}_{\tilde{X}}(\eta\circ \sigma^{-1})\big)\circ \sigma
\\
&=&\big( \nabla^{\perp}_{d\sigma(\tilde{X})} T_{\theta}(\eta \circ \sigma^{-1})\big)\circ \sigma  
 -T_{\theta}\big(\nabla^{\perp}_{\tilde{X}}(\eta\circ \sigma^{-1})\big)\circ \sigma
\\
&=&T_{\theta}\big( \nabla^{\perp}_{d\sigma(\tilde{X})} (\eta \circ \sigma^{-1})\big)\circ \sigma  
 -T_{\theta}\big(\nabla^{\perp}_{\tilde{X}}(\eta\circ \sigma^{-1})\big)\circ \sigma
\\
&=&T_{\theta}\big( \nabla^{\perp}_{d\sigma(\tilde{X})} (\eta \circ \sigma^{-1})  
 -\nabla^{\perp}_{\tilde{X}}(\eta\circ \sigma^{-1})\big)\circ \sigma,
\end{eqnarray*}
where, by abuse of notation, $\nabla^{\perp}$ stands for the normal connection of every involved immersion. 
Since $\eta\circ \sigma^{-1}$ is a section of the normal bundle of $\tilde{f}=f\circ p$, we may write $\eta\circ \sigma^{-1}=\delta\circ p$ 
for some local section $\delta$ of the normal bundle of $f$. We observe that
\begin{eqnarray*}
\nabla^{\perp}_{d\sigma(\tilde{X})} (\eta \circ \sigma^{-1})  
 -\nabla^{\perp}_{\tilde{X}}(\eta\circ \sigma^{-1})
&=&\nabla^{\perp}_{d\sigma(\tilde{X})} (\delta\circ p)  
 -\nabla^{\perp}_{\tilde{X}}(\delta\circ p)\\
&=&\nabla^{\perp}_{dp\circ d\sigma(\tilde{X})} \delta  
 -\nabla^{\perp}_{dp(\tilde{X})}\delta =0.
\end{eqnarray*}
Therefore $\mathit{\Sigma}_{\theta}$ is parallel, orthogonal and preserves the second fundamental forms of $\tilde{f}_{\theta}$ and
  $\tilde{f}_{\theta}\circ \sigma$, and this completes the proof.
\end{proof}

\medskip
Lemma \ref{deck} allows us to define a map 
$$\varPhi_{\theta}:\mathcal{D} \to \text{Isom}(S^4)$$
for each $\theta \in [0,2\pi]$, such that 
$$\tilde{f}_{\theta}\circ \sigma=\varPhi_{\theta}(\sigma) \circ \tilde{f}_{\theta}$$
for any $\sigma \in \mathcal{D}$. 
It is easy to see that $\varPhi_{\theta}$ is a homomorphism for each $\theta \in [0,2\pi]$. Furthermore,  $\theta \in \mathcal{S}(f)$ if and only if
 $\varPhi_{\theta}(\mathcal{D})=\{I\}$. In case where the image of $f$ is contained in a totally geodesic $S^3,$ then $\varPhi_{\theta}$
maps $\mathcal{D}$ into $\text{Isom}(S^3)$.

Now we are ready to prove Theorem 1 which describes the structure of the set $\mathcal{S}(f)$.

\begin{proof}[Proof of Theorem 1]
Assume that $\mathcal{S}(f)$ is infinite. Then there exists a sequence $\{\theta_m\}$ in $\mathcal{S}(f)$ which we may assume converges to some
 $\theta_0 \in [0,2\pi],$ by passing to a subsequence if necessary. From  $\varPhi_{\theta_m}(\mathcal{D})=\{I\}$ for all $m \in \mathbb{N},$ 
we immediately obtain 
 $\varPhi_{\theta_0}(\mathcal{D})=\{I\}$. Let  $\sigma \in \mathcal{D}$. By  applying the Mean Value Theorem to each entry
 $(\varPhi_{\theta}(\sigma))_{jk}$ of the corresponding matrix, we get
$$\dfrac{d}{d\theta}(\varPhi_{\theta}(\sigma))_{jk}(\mathring{\theta}_m)=0$$
for some $\mathring{\theta}_m $ which lies between $\theta_0$ and $\theta_m$. By continuity we get
$$\dfrac{d}{d\theta}(\varPhi_{\theta}(\sigma))_{jk}(\theta_0)=0$$
for each $\sigma \in \mathcal{D}$. Applying repeatedly the  Mean Value Theorem, we conclude that 
$$\dfrac{d^n}{d\theta^n}(\varPhi_{\theta}(\sigma))_{jk}(\theta_0)=0$$
for each $\sigma \in \mathcal{D}$ and any integer $n \geq 1$. Since $\varPhi_{\theta}(\sigma)$ is an analytic curve (cf. \cite{EQ}) in $\text{Isom}(S^4)$, we infer that  
$\varPhi_{\theta}(\sigma)=I$, and so $\mathcal{S}(f)=[0,2\pi]$.
\end{proof}

\begin{lemma}\label{isomorphism}
Let $f:M\rightarrow S^{4}$ be an immersed  minimal surface whose image is not contained in any totally geodesic $S^3$. Then for each $\theta \in \mathcal{S}(f)$ 
there exists a parallel
 and orthogonal bundle isomorphism $T_{\theta}:Nf \to Nf_{\theta} $ such that the second fundamental forms of $f$
 and $f_{\theta}$ are related by
$$B^{f_\theta}(X,Y)=T_{\theta}(B^f(J_{\theta}X,Y))$$
for all $X,Y$ tangent to $M$.
\end{lemma}
\begin{proof}
Let  $\theta \in \mathcal{S}(f)$.  We claim that 
for each point there exist an   open neighborhood $U$   and   a parallel
 and orthogonal bundle isomorphism 
$$T^U_{\theta}:Nf|_U \to Nf_{\theta} |_U$$
 such that the second fundamental forms of $f|_U$
 and $f_{\theta}|_U$  are related by
\begin{equation}\label{BU}
B^{f_{\theta}|_U}(X,Y)=T^U_{\theta}(B^{f|_U}(J_{\theta}X,Y)) 
\end{equation}
for all $X,Y$ tangent to $U$.

Indeed, let $\tilde U \subset \tilde M$  and $U\subset M$ are chosen so that the Riemannian covering map  
$p: \tilde{M} \to M$ maps $\tilde U$ isometrically onto $U$, $\tilde{M}$ being the universal cover of $M$. 
We define the orthogonal isomorphism $T^U_{\theta}$ between the bundles $Nf|_U$ and $Nf_{\theta} |_U $ 
by 
\begin{equation*}
T^U_{\theta}(\xi):=\tilde T_{\theta}(\xi \circ p) \circ (p|_{\tilde U})^{-1}, 
\end{equation*} 
where $\xi$ is section of $Nf|_U$ and 
$$\tilde T_{\theta}:N\tilde f \to N\tilde f_{\theta}$$
is the orthogonal and parallel  isomorphism between the normal bundles of the minimal surfaces $\tilde{f}=f\circ p$ 
and $\tilde{f}_{\theta}=f_{\theta}\circ p$, so that 
\begin{equation}\label{newT}
B^{\tilde f_{\theta}}(\tilde X,\tilde Y)=\tilde T_{\theta}(\tilde B(\tilde J_{\theta}\tilde X,\tilde Y)) 
\end{equation} 
for all $\tilde X,\tilde Y$ tangent to $\tilde M$.

The second fundamental forms of the  minimal surfaces $\tilde{f}=f\circ p$ 
and $\tilde{f}_{\theta}=f_{\theta}\circ p$ are given by
$$   B^{\tilde f}(\tilde X,\tilde Y)=B^{f}(dp(\tilde X), dp(\tilde Y))$$
and 
$$B^{\tilde f_{\theta}}(\tilde X,\tilde Y)=B^{f_{\theta}}(dp(\tilde X), dp(\tilde Y))   
$$
for all $\tilde X $ and $ \tilde Y$. Then it is easy to see that (\ref{newT}) implies (\ref{BU}).

Let  $\xi$ is section of $Nf|_U$. For any $X=dp(\tilde X)$  tangent to $U$, we have 
\begin{eqnarray*}
(\nabla^{\perp}_{ X}T^U_{\theta}) \xi&=&
 \nabla^{\perp}_{ dp(\tilde X)}T^U_{\theta}(\xi)-T^U_{\theta}(\nabla^{\perp}_{ dp(\tilde X)}\xi) \\
&=&\nabla^{\perp}_{dp(\tilde X)}\big(  \tilde T_{\theta}(\xi \circ p|_{\tilde U}) \circ (p|_{\tilde U})^{-1}   \big)
-T^U_{\theta}\big(\nabla^{\perp}_{\tilde X}(\xi \circ p |_{\tilde U})\circ (p|_{\tilde U})^{-1}\big)\\
&=&\big(\nabla^{\perp}_{\tilde X}  \tilde T_{\theta}(\xi \circ p|_{\tilde U})    \big)\circ (p|_{\tilde U})^{-1}
- \tilde T_{\theta}\big(\nabla^{\perp}_{\tilde X} (\xi \circ p|_{\tilde U})    \big)\circ (p|_{\tilde U})^{-1},
\end{eqnarray*}
where, by abuse of notation, $\nabla^{\perp}$ stands for the normal connection of every involved immersion. 
The above shows that $T^U_{\theta}$ is parallel, since $\tilde T_{\theta}$ is parallel.

 Let $V$ be another open subset of $M$ with $U\cap V \neq \varnothing$ and corresponding bundle isomorphism 
$$T^V_{\theta}:Nf|_V \to Nf_{\theta} |_V$$ 
such that the second fundamental forms of $f|_V$
 and $f_{\theta}|_V$   are related by
\begin{equation}\label{BV}
B^{f_{\theta}|_V}(X,Y)=T^V_{\theta}(B^{f|_V}(J_{\theta}X,Y)) 
\end{equation} 
for all $X$ and $Y$ tangent to $V$. We consider the set $M_0$ of points where the normal curvature vanishes, or equivalently, the set of points where the first normal space
 is a proper subset of the normal space. Our assumption  implies that the set $M\smallsetminus M_0$ is dense in $M$.
 From (\ref{BU}) and (\ref{BV}) we see that $T^U_{\theta}=T^V_{\theta}$ on $U\cap V \smallsetminus M_0$.
By continuity,  
we infer that $T^U_{\theta}=T^V_{\theta}$ on $U\cap V $. Thus $T^U_{\theta}$ is globally well-defined.
\end{proof}

\section{Proof of the main result}

Before we proceed to the proof of Theorem 2, we recall  some useful facts.
At first, we need the topological restrictions for minimal surfaces in $S^4$ that were obtained by Eschenburg and Tribuzy \cite{ET0}. 
To this purpose, we review some properties of absolute value type  functions.

The zero set  of an absolute value type function $a$ on a connected compact oriented surface $M$ is either isolated or the whole of $M$, and outside 
its zeros, the function is smooth.
If $a$ is a nonzero absolute value type function, i.e., locally $a=|t_{0}|a_{1}$, with $t_0$ holomorphic, the order $k\geq 1$ of any $x \in M$
 with $a(x)=0$ 
is the order of $t_0$ at $x$. Let $N(a)$ be the sum of all orders for all zeros of $a$. Then $\Delta \log a$ is bounded on
 $M\smallsetminus \{ a=0  \}$ and its integral is given by
\begin{equation*}
\int_{M}\Delta \log adA=-2\pi N(a).
\end{equation*}

The following lemma,  due to Eschenburg and Tribuzy \cite{ET0}, follows immediately from Theorem \ref{ET0} 
just by integrating (\ref{Laplace}) and using the Gauss-Bonnet Theorem and the fact that the Euler number $\chi (Nf)$ 
of the normal bundle is given by
\begin{equation*}
 \int_M K_N dA=2\pi \chi (Nf).
\end{equation*}

\begin{lemma}\label{x}
Let $f:M\rightarrow S^{4}$ be a compact oriented immersed minimal surface. If $f$ is not superminimal, then the Euler number $\chi (Nf)$ 
of the 
normal bundle and the Euler-Poincar\'{e} characteristic $\chi (M)$ of $M$ satisfy 
$$2\chi (M)\pm \chi (Nf)=-N(a_{\mp}).$$
\end{lemma}

We also need some facts about holomorphic bundle-valued forms (cf.  \cite{BWW}). Let $M$ be a 
2-dimensional oriented Riemannian manifold with the canonically defined complex structure and $\mathsf{E}$ be a complex vector bundle over $M$
 equipped with a connection ${\nabla}$. For any $\mathsf{E}$-valued $r$-covariant tensor field $F$ on $M$ 
the covariant derivative is defined in the usual way, where $M$ is equipped with the Levi-Civit\'{a} connection. 
If $F$ is of holomorphic type $(r,0)$, we say that $F$ is holomorphic if its covariant derivative is of holomorphic type $(r+1,0)$. In terms of a local
complex coordinate $(U,z)$ on $M$, a tensor field $F$  of holomorphic type $(r,0)$ is written on $U$  in the form 
$$F=udz^r,$$ 
where $u:U \to \mathsf{E}$ is given by
$$u=F\big(\dfrac{\partial}{\partial z}, \dots,   \dfrac{\partial}{\partial z}   \big).$$
Then $F$ is holomorphic if and only if 
$${\nabla}_{\frac{\partial}{\partial \overline{z}}}u=0.$$
The following, which we quote from \cite{BWW}, is crucial for the proof of the main result.

\begin{theorem}\label{BWW}
Assume that the  $\mathsf{E}$-valued  tensor field $F$ on $M$ is holomorphic, and let $x \in M$ be such that $F(x)=0$. Let $(U,z)$ be a local
complex coordinate on $M$ with $z(x)=0.$ Then either $F \equiv 0$ on $U$, or $F=z^mF^*,$ where $m$ is a positive integer and $F^*(x)\neq 0.$ 
\end{theorem}

Now let $f:M\rightarrow S^{4}$ be an immersed minimal surface which is not contained in any totally geodesic $S^3$. Assume hereafter that $f$ is not superminimal. 
The set  $M_1$ of  points where the curvature ellipse 
is  a circle consists of isolated points only. For each point  $x \in M\smallsetminus M_1$, 
we consider an orthonormal frame $\{e_1,e_2,e_3, e_4 \}$  on a 
neighborhood $U_x\subset M\smallsetminus M_1$ of $x$ as
 in Lemma \ref{local theory} with normal connection form $\omega_{34}$.

For arbitrary  $\theta_j \in \mathcal{S}(f),j=1,\dots, n, $  we consider the local orthonormal frame
 $\{e^{\theta_j}_{3}, e^{\theta_j}_{4} \}$ of the normal bundle of $f_{\theta_j}$ 
defined by 
$$e^{\theta_j}_{3}:=T_{\theta_j}(e_3), {\ } e^{\theta_j}_{4}:=T_{\theta_j}(e_4),$$ 
where 
$$T_{\theta_j}:Nf \to Nf_{\theta_j} $$
 is the bundle isomorphism  of Lemma \ref{isomorphism}. Obviously, $\omega_{34}$ 
 is also  the normal connection form of $f_{\theta_j}$
with respect to this frame. By virtue of Lemma \ref{isomorphism}, we easily find  that $H_3 , H_4$ and the corresponding functions
 $H^{\theta_j}_{3} , H^{\theta_j}_{4}$ for $f_{\theta_j}$ are related by 
\begin{equation}\label{Htheta}
 H^{\theta_j}_{3}=\exp(-2i\theta_j) H_3 {\ }{\ } \text{and} {\ }{\ } H^{\theta_j}_{4}=\exp(-2i\theta_j) H_4.
\end{equation}
 Using (\ref{Htheta}) and the Weingarten formula for $f_{\theta_j}$, we obtain 
\begin{equation}\label{Weingarten1}
  D_{{E}}e^{\theta_j}_{3}=-\kappa_1\exp(i\theta_j) df_{\theta_j}(\overline{E})+ \omega_{34}({E})e^{\theta_j}_{4}
\end{equation} 
and
\begin{equation}\label{Weingarten2}
 D_{{E}}e^{\theta_j}_{4}=i\mu_1 \exp(i\theta_j) df_{\theta_j}(\overline{E})- \omega_{34}(E)e^{\theta_j}_{3},
\end{equation} 
where 
$$E=e_1-ie_2$$
 and $D$ stands for the usual connection in the induced bundle $(i_1\circ f)^*(T\mathbb{R}^5)$, $i_1: S^4 \to \mathbb{R}^5$ being the inclusion map.

The following auxiliary lemma is needed for the proof of the main result.

\begin{lemma}\label{all}
Assume that there exist vectors $v_j \in \mathbb{R}^5,j=1,\dots, n,$ such that
\begin{equation}\label{coordinate}
\sum_{j=1}^n \langle  f_{\theta_j}, v_j\rangle =0  {\ }{\ } \text{on} {\ }{\ } U_x . 
\end{equation}

(i) Then 
\begin{equation}\label{kappa-mu}
\sum_{j=1}^n \exp(i\theta_j) \Big(\kappa_1 \langle     e^{\theta_j}_{3},   v_j   \rangle -i\mu_1 \langle    e^{\theta_j}_{4},  v_j     \rangle \Big)=0.
\end{equation}

(ii) Away from points where $\omega_{34}$ vanishes, we have
\begin{eqnarray}\label{basic}
 \sum_{j=1}^n \exp(i\theta_j) \langle     e^{\theta_j}_{4},   v_j   \rangle =
\dfrac{\kappa_1}{2\omega_{34}({E})}\sum_{j=1}^n \exp(2i\theta_j)\langle    df_{\theta_j}(\overline{E}),  v_j    \rangle
\end{eqnarray}
and
\begin{equation}\label{crucial}
2\omega_{34}({E})\sum_{j=1}^n \exp(2i\theta_j) \langle    f_{\theta_j},  v_j    \rangle= 
L  \sum_{j=1}^n \exp(2i\theta_j) \langle    df_{\theta_j}(\overline{E}),   v_j  \rangle,
\end{equation} 
where $L$ is the complex valued function given by 
$$L=-{E}\big(\omega_{34}({E})\big)-3i \omega_{12}({E})\omega_{34}({E}).$$

(iii) Furthermore, we have
\begin{eqnarray}\label{holomorphic}
\overline{E}\Big( \sum_{j=1}^n \exp(i\theta_j) \langle     e^{\theta_j}_{4},   v_j   \rangle \Big)=
-\omega_{34}(\overline{E})\sum_{j=1}^n \exp(i\theta_j) \langle    e^{\theta_j}_{3},  v_j     \rangle.
\end{eqnarray}
\end{lemma}

\begin{proof}
Our assumption implies that
\begin{equation}\label{d}
\sum_{j=1}^n \langle  df_{\theta_j}, v_j\rangle =0. 
\end{equation}
Differentiating, using the Gauss formula and (\ref{coordinate}), we immediately see
that the second fundamental forms of $f_{\theta_j}$ satisfy 
\begin{equation*}
\sum_{j=1}^n \langle  B^{f_{\theta_j}}, v_j\rangle =0. 
\end{equation*}
This on account of (\ref{Htheta}) yields
\begin{equation*}
\sum_{j=1}^n \exp(i\theta_j) \Big(\overline{H}_3 \langle  e^{\theta_j}_{3}, v_j   \rangle +\overline{H}_4  \langle    e^{\theta_j}_{4},  v_j \rangle \Big)=0. 
\end{equation*}
Since the frame  $\{e_{3}, e_{4} \}$ is chosen as in Lemma \ref{local theory}, we have $H_3  ={\kappa_1}, H_4= i {\mu_1 }$, and the above
 immediately implies (\ref{kappa-mu}).

Differentiating (\ref{kappa-mu}) with respect to ${E}$ and using (\ref{Weingarten1}) and (\ref{Weingarten2}) we obtain
\begin{eqnarray*}{}
\sum_{j=1}^n \exp(i\theta_j) \Big(\big( E(\kappa_1)+i\mu_1 \omega_{34}({E}) \big) \langle  e^{\theta_j}_{3}, v_j   \rangle 
- \big( iE(\mu_1)-\kappa_1 \omega_{34}({E})  \big)  \langle    e^{\theta_j}_{4},  v_j \rangle \Big)
\\
=(\kappa_1^2-\mu_1^2)\sum_{j=1}^n \exp(2i\theta_j) \langle    df_{\theta_j}(\overline{E}),   v_j  \rangle.  
\end{eqnarray*}
We view (\ref{kappa-mu}) and the above equation as a linear system with unknowns 
$$\sum_{j=1}^n \exp(i\theta_j)\langle    e^{\theta_j}_{3},  v_j     \rangle, {\ } 
\sum_{j=1}^n \exp(i\theta_j)\langle    e^{\theta_j}_{4},  v_j     \rangle.$$
The determinant of this system  is given by
$$a= (\kappa_1^2-\mu_1^2)\omega_{34}({E})-i\big(\kappa_1 E(\mu_1)-\mu_1 E(\kappa_1)  \big).$$
Equations (\ref{connection}) yield
\begin{equation}\label{kappa}
 E(\kappa_1)=-2i\kappa_1\omega_{12}(E)+i\mu_1\omega_{34}({E})
\end{equation}
and
\begin{equation}\label{mu}
 E(\mu_1)=-2i\mu_1\omega_{12}(E)+i\kappa_1\omega_{34}({E}).
\end{equation}
Then  the determinant is written as 
$$a=2(\kappa_1^2-\mu_1^2)\omega_{34}({E}).$$
Using (\ref{kappa}), (\ref{mu}) and solving the linear system,   we  easily obtain (\ref{basic}).

Now differentiating (\ref{basic})  with respect to $E$,  using  (\ref{Weingarten1}), (\ref{Weingarten2}), (\ref{kappa-mu}),  (\ref{basic}) and
the Gauss formula, we find
\begin{eqnarray*}
 \Big( E\big(  \dfrac{\kappa_1}{\omega_{34}(E)}     \big)  -\dfrac{i\kappa_1\omega_{12}(E)}{\omega_{34}(E)} -i\mu_1   \Big)
\sum_{j=1}^n \exp(2i\theta_j) \langle    df_{\theta_j}(\overline{E}),   v_j  \rangle\\
=-\dfrac{2\kappa_1}{\omega_{34}(E)}\sum_{j=1}^n \exp(2i\theta_j)\langle    f_{\theta_j},  v_j    \rangle,
\end{eqnarray*}
or equivalently, by virtue of (\ref{kappa}),
\begin{eqnarray*}
\kappa_1\Big({E}\big(\omega_{34}({E})\big)+3i \omega_{12}({E})\omega_{34}({E}) \Big)   
   \sum_{j=1}^n \exp(2i\theta_j) \langle    df_{\theta_j}(\overline{E}),   v_j  \rangle \\
=
-2\kappa_1\omega_{34}({E})\sum_{j=1}^n \exp(2i\theta_j)\langle    f_{\theta_j},  v_j    \rangle.  
\end{eqnarray*}
This last equation is equivalent to (\ref{crucial}), since $\kappa_1^2>\mu_1^2$ on $U_x.$

Appealing to (\ref{Weingarten2}), we observe that
\begin{eqnarray*}
\overline{E}\Big( \sum_{j=1}^n \exp(i\theta_j)\langle     e^{\theta_j}_{4},   v_j   \rangle \Big)
&=&-i\mu_1 \sum_{j=1}^n \langle df_{\theta_j}(E), v_j   \rangle \\
&-& \omega_{34}(\overline{E}) 
 \sum_{j=1}^n \exp(i\theta_j)\langle   e^{\theta_j}_{3},  v_j   \rangle, 
\end{eqnarray*}
which in view of (\ref{d}) immediately yields (\ref{holomorphic}). 
\end{proof}

Now we are ready to give the proof of the main result.

\begin{proof}[Proof of Theorem 2]
 Let $f:M \to S^4$ be an isometric minimal immersion of a compact oriented 2-dimensional Riemannian manifold $M$
 into $S^4$ with nontrivial normal bundle. 
We may assume that $f$ is not superminimal, otherwise there is nothing to prove. 
According to Theorem 1, either there are only finitely many  noncongruent immersed minimal surfaces isometric to $f$ with the same normal curvature,
 or  the space of all minimal surfaces in $S^4$ with these properties is a circle.

Arguing indirectly, we suppose that $\mathcal{S}(f)=[0,2\pi]$. The strategy is to prove that the coordinate functions of the minimal surfaces 
$f_{\theta}, \theta \in [0,2\pi],$ are linearly independent. On the other hand, these functions are eigenfunctions of the Laplace
operator of $M$ with corresponding eigenvalue 2. This leads to a contradiction since the eigenspaces of the Laplace operator are finite dimensional.

\begin{claim}
Let $0<\theta_1<\dots<\theta_n\leq 2\pi$.  If  for vectors $v_j \in \mathbb{R}^5,j=1,\dots, n,$ the following holds 
\begin{equation}\label{vector}
\sum_{j=1}^n \langle  f_{\theta_j}, v_j\rangle =0, 
\end{equation} 
then $v_j=0$ for all $j=1,\dots, n.$
\end{claim}

Assume to the contrary that each vector $v_j,j=1,\dots, n,$  is nonzero. Let $M_1$ be the set of isolated points where the curvature ellipse 
is  a circle. Obviously $M_1$ is finite. We set $M_1=\{x_1, \dots,x_k\}$. 
Around each point  $x \in M\smallsetminus M_1$, we consider a local complex coordinate $(U_x, z)$ and an 
 orthonormal frame $\{e_1,e_2,e_3, e_4 \}$ along $f$ on  $U_x\subset M\smallsetminus M_1$ as
 in Lemma \ref{local theory} with normal connection form $\omega_{34}$.

We also consider the local orthonormal frame
 $\{e^{\theta_j}_{3}, e^{\theta_j}_{4} \}$ of the normal bundle of $f_{\theta_j}$ 
defined by 
\begin{equation*}\label{thetaframe}
e^{\theta_j}_{3}:=T_{\theta_j}(e_3), {\ } e^{\theta_j}_{4}:=T_{\theta_j}(e_4), 
\end{equation*} 
where 
$$T_{\theta_j}:Nf \to Nf_{\theta_j} $$
 is the bundle isomorphism  of Lemma \ref{isomorphism}.

We now define the smooth complex valued functions $\varphi, \psi: U_x \to \mathbb{C}$ by
\begin{equation*}
\varphi:= \sum_{j=1}^n \exp(i\theta_j)\langle e^{\theta_j}_{3}, v_j \rangle {\ }{\ } \text{and} {\ } {\ }
\psi:= \sum_{j=1}^n \exp(i\theta_j)\langle e^{\theta_j}_{4}, v_j \rangle.
\end{equation*}
These functions are locally defined and obviously depend on the choice frame $\{ e_3,e_4\}$. If for another point $\hat{x} \in M\smallsetminus M_1$, 
with corresponding
neighborhood $U_{\hat{x}}\subset M\smallsetminus M_1$ and frame $\{ \hat{e}_3,\hat{e}_4\}$ chosen as in Lemma \ref{local theory}, we have 
$U_x\cap U_{\hat{x}} \neq \varnothing$, then either $\hat{e}_3=e_3$ and $\hat{e}_4=e_4$ or $\hat{e}_3=-e_3$ and $\hat{e}_4=-e_4$ on 
$U_x\cap U_{\hat{x}}.$ The corresponding functions 
$\hat\varphi, \hat\psi: U_{\hat x} \to \mathbb{C}$ are given by
\begin{equation*}
\hat\varphi:= \sum_{j=1}^n \exp(i\theta_j)\langle \hat e^{\theta_j}_{3}, v_j \rangle {\ }{\ } \text{and} {\ } {\ }
\hat\psi:= \sum_{j=1}^n \exp(i\theta_j)\langle \hat e^{\theta_j}_{4}, v_j \rangle,
\end{equation*}
where
\begin{equation*}\label{thetaframe}
\hat e^{\theta_j}_{3}:=T_{\theta_j}(\hat e_3), {\ } \hat e^{\theta_j}_{4}:=T_{\theta_j}(\hat e_4). 
\end{equation*} 
It is obvious that $ \varphi^2 = \hat\varphi^2$ and $ \psi^2 = \hat\psi^2$ on $U_x\cap U_{\hat{x}}$.
 This means that the functions $\varphi^2, \psi^2$ are globally well-defined on $M\smallsetminus M_1$.

From (\ref{kappa-mu}) we have
\begin{equation*}
 \varphi=\dfrac{i\mu_1}{\kappa_1}\psi,
\end{equation*}
while (\ref{connection}) yields
\begin{equation*}
 \omega_{34}(\overline{E})=\dfrac{i}{\kappa_1^2-\mu_1^2}\big(  \kappa_1  \overline{E} (\mu_1) - \mu_1\overline{E}(\kappa_1) \big).
\end{equation*}
Then using (\ref{holomorphic}) and the last two equations, we easily verify that 
\begin{equation*}
\overline{E}\Big( \psi^2 \big(1-\dfrac{\mu_1^2}{\kappa_1^2}\big)  \Big) =0,
\end{equation*}
or equivalently,
\begin{equation*}
 \dfrac{\partial}{\partial \bar z}   \Big( \psi^2 \big(1-\dfrac{\mu^2}{\kappa^2}\big)  \Big) =0,
\end{equation*}
where $z$ is the local complex coordinate of  Lemma \ref{local theory}.
This shows that the function 
$$\psi^2 \big(1-\dfrac{\mu^2}{\kappa^2}\big): M\smallsetminus M_1 \to \mathbb{C}$$
 is holomorphic with isolated singularities. From the inequality
\begin{equation*}
 \Big|\psi^2 \big(1-\dfrac{\mu^2}{\kappa^2}\big)\Big| \leq \Big( \sum_{j=1}^n|v_j|  \Big)^2,
\end{equation*}
 we see that it is also bounded. Thus its singularities are removable and we end up with a 
$\mathbb{C}$-valued holomorphic function on  the compact Riemann surface $M$.  Hence there exists a constant $c \in \mathbb{C}$ such that
\begin{equation}\label{c}
\psi^2 \big(\kappa^2-\mu^2\big)=c\kappa^2 {\ }{\ } \text{on} {\ }{\ } M\smallsetminus M_1.
\end{equation} 

We will show that $c=0$. Indeed if there exists a point $x_l \in M_1$ with $\kappa(x_l)=\mu(x_l)>0,$ then taking the limit in (\ref{c}) along a sequence
of points  in $ M\smallsetminus M_1$ which converges to $x_l$ 
and using the boundedness of $\psi^2$, we deduce that $c=0$. 

Suppose now that for all points in $M_1$ the curvature ellipse degenerates into a point, i.e., 
$\kappa(x_l)=\mu(x_l)=0$ for all $l=1,\dots,k.$ In others words, all points in $M_1$ are totally geodesic points.  For each $l=1,\dots,k,$ 
let $(V,z)$ be a local
complex coordinate around $x_l$ with $z(x_l)=0.$ It is a well known consequence of the Codazzi equation that the $(2,0)$-part
 $$B^{(2,0)}=B\big(\dfrac{\partial}{\partial z},\dfrac{\partial}{\partial z} \big)dz^2$$ of the   second fundamental form of $f$ is holomorphic
as a $Nf\otimes\mathbb{C}$-valued tensor field (cf. \cite{GR} or \cite{Ru}). Since $B^{(2,0)}$  is not identically zero and $x_l$ is a zero of it, 
according to Theorem \ref{BWW}, we may write
\begin{equation}\label{m}
B^{(2,0)}=z^{m_l} B^{*(2,0)} {\ }{\ } \text{on} {\ }{\ } V
\end{equation} 
for a positive integer $m_l$, where $B^{*(2,0)}$ is a tensor field on $V$ of type $(2,0)$ with $B^{*(2,0)}|_{x_l}\neq 0$.
We now define the $Nf$-valued tensor field on $V$
\begin{equation*}
 B^*:=B^{*(2,0)}+\overline{B^{*(2,0)}}.
\end{equation*}
Since the $(1,1)$-part of $B^*$ vanishes, it follows easily that
$B^*$ maps the unit tangent circle at each tangent plane  on $V$ into an ellipse on the corresponding normal space
with  length of the semi-axes  $\kappa^* \geq \mu^*\geq 0$. 

We also consider the differential form of type $(4,0)$ 
\begin{equation*}
\Phi^*:=\langle B^{*(2,0)},B^{*(2,0)}\rangle
\end{equation*}
which, in view of (\ref{m}), is related to the Hopf differential  of $f$ by 
\begin{equation*}
\Phi=z^{2m_l}\Phi^*. 
\end{equation*}
We now consider  arbitrary orthonormal frames  $\{\xi _{1},\xi_2\}$ and $\{\xi _{3},\xi_4\}$  of $TM|V$ and $Nf|V$
respectively, that agree with the given orientations.   Then we split the form $\Phi^*$, with respect to this frame,  in the same manner we spitted
 the Hopf differential in section 2, i.e.,
\begin{equation*}
\Phi =\dfrac{1}{4}\big( {\overline{H}_{3}^{2}}+{\overline{H}%
_{4}^{2}}\big) \varphi ^{4}=\dfrac{1}{4}k_{+}k_{-}\varphi^{4} 
\end{equation*}
and  
\begin{equation*}
\Phi^* =\dfrac{1}{4}\big( {\overline{H}_{3}^{*2}}+{\overline{H}%
_{4}^{*2}}\big) \varphi ^{4}=\dfrac{1}{4}k^*_{+}k^*_{-}\varphi^{4}, 
\end{equation*}
where 
\begin{equation*}
k_{\pm}= {\overline{H}_{3}} \pm i{\overline{H}%
_{4}}, {\ }{\ }  k^*_{\pm}= {\overline{H}_{3}^*} \pm i{\overline{H}^*
_{4}}
\end{equation*}
and 
\begin{equation*}
H_{\alpha }=h_{1}^{\alpha }+ih_{2}^{\alpha },  {\ }{\ } H^*_{\alpha }=h_{1}^{*\alpha }+ih_{2}^{*\alpha }.
\end{equation*}
The components of $B$ and $B^*$ are given respectively by
 \begin{equation*}
h_{1}^{\alpha }=\langle
B(\xi_{1},\xi_{1}),\xi_{\alpha }\rangle, {\ } h_{2}^{\alpha }=\langle
B(\xi_{1},\xi_2),\xi_{\alpha }\rangle
\end{equation*}
and
\begin{equation*}
h_{1}^{*\alpha }=\langle
B^*(\xi_{1},\xi_{1}),\xi_{\alpha }\rangle, {\ } h_{2}^{*\alpha }=\langle
B^*(\xi_{1},\xi_2),\xi_{\alpha }\rangle
\end{equation*}
 for $\alpha =3,4.$ 
Then, in view of (\ref{m}), we obtain 
$\overline{H}_{\alpha }=z^{m_l}\overline{H}^*_{\alpha },$
 or 
$k_{\pm}=z^{m_l}k^*_{\pm}.$ Hence 
\begin{equation*}\label{a}
a_{\pm}=|z|^{m_l}a^*_{\pm}, 
\end{equation*} 
or equivalently,
\begin{equation}\label{semi}
\kappa\pm \mu=|z|^{m_l}(\kappa^{*}\pm \mu^{*}). 
\end{equation} 
From this we deduce that 
\begin{equation*}\label{kappamu*}
 \kappa=|z|^{m_l}\kappa^{*} {\ }{\ } \text{and} {\ }{\ } \mu=|z|^{m_l}\mu^{*} 
\end{equation*}
and (\ref{c}) now yields 
\begin{equation}\label{c*}
\psi^2 \big(\kappa^{*2}-\mu^{*2}\big)=c\kappa^{*2} {\ }{\ } \text{on} {\ }{\ } V\smallsetminus\{x_l\}. 
\end{equation}

To prove that $c=0$, we now argue in the following way. If $\kappa^{*}(x_l)>\mu^{*}(x_l)$ for all $1\leq l \leq k$, then (\ref{semi}) implies that 
\begin{equation*}
 N(a_+)=\sum_{l=1}^{k}m_l=N(a_-). 
\end{equation*}
Hence Lemma \ref{x} yields $\chi(Nf)=0$, which contradicts our assumption. 
Thus  $\kappa^{*}(x_l)=\mu^{*}(x_l)$ for some $1\leq l \leq k$. Taking the limit in (\ref{c*}),
 along a sequence of points in $V\smallsetminus\{x_l\}$ which converges to $x_l$, we obtain $c\kappa^{*2}(x_l)=0$. 
Obviously $\kappa^{*}(x_l)>0$, since $B^{*}|_{x_l}\neq 0,$ and consequently we infer that $c=0$.

From (\ref{c}), we conclude that $\psi^2=0$ everywhere on $M\smallsetminus M_1$. We note that $\omega_{34}(E)$ cannot vanish on an open subset of
 $M\smallsetminus M_1$.
Indeed, if $\omega_{34}(E)=0$ on an open subset $U\subset M\smallsetminus M_1$, then (\ref{connection}) 
would imply that the ratio $\mu_1/\kappa_1$ is constant on $U$. Using  Theorem \ref{ET0}, it is easy to see that the 
Gaussian curvature $K$ satisfies the Ricci condition 
$$\varDelta \log (1-K)=4K$$
on $U$.
 According to a result due to Sakaki \cite{S}, and bearing in mind the fact that minimal surfaces are real analytic, we infer that our minimal 
surface lies in a totally geodesic $S^3$, which contradicts the assumption on  the normal bundle.

Hence we may appeal to (\ref{basic}) and  (\ref{crucial}) to obtain
\begin{equation*}
\sum_{j=1}^n \exp(2i\theta_j)\langle  f_{\theta_j}, v_j\rangle =0.  
\end{equation*}
Combining this with (\ref{vector}), we get
\begin{equation*}\label{v1}
\sum_{j=2}^n \langle  f_{\theta_j}, w_j\rangle =0,  
\end{equation*}
where $w_j:=\lambda_jv_j\neq 0, j=2, \dots, n,$ and $\lambda_j$ is either $\cos 2\theta_n-\cos 2\theta_1$ or $\sin 2\theta_n-\sin 2\theta_1$. 
Then repeating  the same argument, we inductively   conclude at the end that $\langle  f_{\theta_n}, w\rangle =0$, for some nonzero vector $w$.
 So $f_{\theta_n}$
lies in a totally geodesic $S^3$, contradiction.

Therefore we have proved Claim 1. This means that the coordinate functions of 
$f_{\theta}, \theta \in [0,2\pi],$ are linearly independent and the proof of the theorem is complete.
\end{proof}

\begin{proof}[Proof of Corollary 1]
From Theorem 2, we know that 
$$\mathcal{S}(f)=\{\theta_0, \theta_1,\dots,\theta_n\}$$
 for some positive integer $n$ with $0=\theta_0<\theta_1<\dots <\theta_n\leq 2\pi$. So $f_{\theta_j},j=1,\dots,n,$ is the maximal
family of  noncongruent minimal surfaces in $S^4$ which are isometric to $f$ and have the same normal curvature. We consider
 the immersed minimal surfaces $f_t:=f \circ \varphi_t$. From the assumptions it follows that each $f_t$ is isometric to $f$
 and has the same normal curvature with $f$. According to Theorem 2 and since the second fundamental form of $f_t$ depends
 continuously on the parameter, we deduce that $f_t$ is congruent to exactly one $f_{\theta_j}$ for all $t$. Since $f\circ\varphi_0=f,$ we conclude that 
$f_t$ is congruent to $f$ for all $t$.
\end{proof}

\end{document}